\documentclass[11pt]{article}

\usepackage{graphicx}
\usepackage[T1]{fontenc}
\usepackage[latin1]{inputenc}
\usepackage{enumerate}
\usepackage{pstricks}
\usepackage{latexsym,amssymb,amsthm,amsxtra}
\usepackage{amsmath}
\usepackage{tikz}
\usepackage{pgf}
\usepackage[numbers,sort]{natbib}
\usepackage{hyperref}

\def\epsilon{\varepsilon}

\usepackage{anysize}
\marginsize{3cm}{2.5cm}{2.5cm}{2.5cm} %{margen_izq}{margen_dch}{margen_sup}{margen_inf}
\newtheorem{theorem}{Theorem}[section]
\newtheorem{lemma}[theorem]{Lemma}

\newtheorem{prop}[theorem]{Proposition}

\newtheorem{remark}[theorem]{Remark}

\newcommand{\be}{ \begin{equation}}
\newcommand{\ee}{\end{equation}}

\def\E{{\mathbb E}}

\def\P{{\mathbb P}}
\def\Q{{\mathbb Q}}
\def\R{{\mathbb R}}

\def\F{{\mathcal{F}}}

\def\L{{\mathcal{L}}}

\def\ve{{\varepsilon}}

\def\({{\Bigl(}}
\def\){{\Bigr)}}

\def\one{{\mathbf 1}}

\def\square{\ifmmode\sqr\else{$\sqr$}\fi}
\def\sqr{\vcenter{
         \hrule height.1mm
         \hbox{\vrule width.1mm height2.2mm\kern2.18mm\vrule width.1mm}
         \hrule height.1mm}}                  % This is a slimmer sqr.

\title{Front propagation and quasi-stationary distributions for one-dimensional L\'evy processes}
\author{Pablo Groisman\thanks{Departamento de Matem\'atica, FCEN, Universidad de Buenos Aires, IMAS-CONICET and NYU-ECNU Institute of Mathematical Sciences at NYU Shanghai. {\tt
pgroisma@dm.uba.ar}, {\tt http://mate.dm.uba.ar/$\sim$pgroisma.}}
  \ and  Matthieu Jonckheere\thanks{Instituto de C\'alculo, FCEN, Universidad de Buenos Aires and IMAS-CONICET. {\tt mjonckhe@dm.uba.ar},
 {\tt http://matthieujonckheere.blogspot.com}.
} }

\date{}
\begin{document}
\maketitle

\abstract{

We jointly investigate the existence of quasi-stationary distributions for one dimensional L\'evy processes and the existence of traveling waves for the Fisher-Kolmogorov-Petrovskii-Piskunov (F-KPP) equation associated with the same motion. 
Using probabilistic ideas developed by S. Harris \cite{SH}, we show that the existence of a traveling wave for the F-KPP equation associated with a centered L\'evy processes that branches at rate $r$ and travels at velocity $c$ is equivalent to the existence of a quasi-stationary distribution for a L\'evy process with the same movement but drifted by $-c$ and killed at zero, with mean absorption time $1/r$. This also extends the known existence conditions in both contexts.  
As it is discussed in \cite{GJ2}, this is not just a coincidence but the consequence of a relation between these two phenomena.

\bigskip



\section{Introduction}
Let $\L$ be the generator of a centered one-dimensional L\'evy process (precise definitions and assumptions are be given below) and consider the (generalized) F-KPP equation
\begin{equation}
 \label{KPP.Levy}
\begin{array}{l}
\displaystyle \frac{\partial u}{\partial t}= \displaystyle  \L^* u + r (u^2 -u), \quad x \in \R,\,
t>0,\\
\\
\displaystyle u(0,x)= \displaystyle u_0(x), \quad x \in \R.
\end{array}
\end{equation}
Here $\L^*$ denotes the adjoint of $\L$. Both Fisher and Kolmogorov, Petrovskii and Piskunov considered this equation for $\L = \frac{d^2}{dx^2}$ and proved independently that in this case this equation admits traveling wave solutions of the form $u(t,x) = w_c(x-ct)$ that travel at velocity $c$ for every $c\ge \sqrt{2r}$, \cite{Fisher, KPP}.

It is well known \cite{B2, Langer, Saarloos, BD} that a large class of equations describing the
propagation of a front into an unstable region have properties similar to \eqref{KPP.Levy}. These equations admit traveling-wave solutions for any velocity $c$ larger than a minimal velocity $c^*$ and the front moves with this minimal velocity $c^*$ for any initial data with ``light enough'' tails.

For the Brownian case  $\L = \frac{d^2}{dx^2}$ we have $c^*=\sqrt{2r}$ and for more general $\L$ the minimal velocity can be computed in terms of the Legendre transform of the process (see Theorem \ref{theo:main} below). This was essentially done by Kyprianou \cite{Ky} using the seminal McKean's representation \cite{MK} for the solutions of \eqref{KPP.Levy}. We complete this characterization in this note to arrive to our main theorem. 
% A traveling wave (TW) for \eqref{KPP.Levy} with speed $c$ for a given branching rate $r$ is a nontrivial solution of
% \be \label{eq:TWL}
%  \L^* w + c w'+ rw(w-1)= 0.
% \ee

The theory of quasi-stationary distributions has its own counterpart.
It is a typical situation that there is an infinite number of quasi-stationary
distributions while the {\em Yaglom limit} (the limit of the conditioned
evolution of the process started from a deterministic initial condition)
selects the minimal one, i.e. the one with minimal expected time of absorption \cite{FMP2,VD,cavender}.

To be more precise, consider a L\'evy process $(X_t - ct)_{t\ge 0}$ with generator $\L - c\frac{d}{dx}$ killed at the origin defined in certain filtered space $(\Omega, \F, (\F_t),\P)$ with expectation denoted by $\E$. The absorption time is defined by $\tau=\inf\{t>0 \colon X_t -ct
= 0\}$. The conditioned evolution at time $t$ is defined by
\[
 \mu_t^\gamma(\cdot):= \P^\gamma(X_t -ct\in  \cdot|\tau>t).
\]

Here $\gamma$ denotes the initial distribution of the process and $\P^\gamma(\cdot)=\P(\cdot|X_0\sim\gamma)$. A probability measure
$\nu$ is
said to be a quasi-stationary distribution (QSD) if $\mu_t^\nu = \nu$ for all
$t\ge 0$.
%
%For Markov chains in finite state spaces, the existence and uniqueness of QSDs as well as the convergence of the conditioned evolution to this unique QSD
%for every initial measure follows from Perron-Frobenius theory.
%The situation is more delicate for unbounded spaces as there can be $0$, $1$ or an infinite number of QSD.
%Among those distributions, the {\em minimal} QSD is the one that minimizes
%$\E_\nu(\tau)$.

The {\em Yaglom limit} is a probability measure $\nu$ defined by
\[
\nu:= \lim_{t \to \infty} \mu_t^{\delta_x},
\]
if the limit exists and does not depend on $x$. It is known that if the Yaglom limit
exists, then it is a QSD. A general principle is that the Yaglom limit
{\em selects} the minimal QSD, i.e. the Yaglom limit is the QSD with minimal mean absorption time. This fact has been proved for a
wide class of processes that include birth and death process,
subcritical Galton-Watson
processes, drifted random walks and Brownian motion among others, but the conjecture is still open for a much wider class of processes.

In the last decades, a great deal of attention has been given to establish on the one hand conditions for the existence of quasi-stationary measures of L\'evy processes (see for instance \cite{MPSM,KP1}) and on the other hand to the existence of traveling waves for \eqref{KPP.Levy} \cite{Ky}. The purpose of this note is to show that given parameters $r,c >0$, the existence of a traveling wave for \eqref{KPP.Levy} with velocity $c$ is equivalent to the existence of a QSD $\nu$ for $\L - c\frac{d}{dx}$ with expected absorption time $\E_\nu(\tau)=1/r$. Moreover, minimal velocity TWs are in a one-to-one correspondence with minimal absorption time QSDs with the same parameters. Note that when dealing with traveling-waves the branching rate $r$ is an input while the velocity $c$ is chosen by the system, while when dealing with QSDs the velocity $c$ is the input and $r$ is chosen by the system.

Although our proof consists in showing that the conditions for the existence of TW and QSD coincide, in a companion paper \cite{GJ2} we show that these is not just a coincidence but that the two phenomena are essentially two faces of the same coin.

All in all, our main result reads. 

\begin{theorem}
\label{theo:main}
Under assumption  {\bf A} (stated below), the following are equivalent:
\begin{enumerate}
\label{TWKPP}
\item  There exists a non trivial traveling wave for \eqref{KPP.Levy} with velocity $c$, i.e. a solution to
\begin{equation}
\label{eq.tw.levy}
\L^* w + c w'+ r w (w-1)=0.
\end{equation}
\item There exists an (absolutely continuous) QSD for $\L - c \frac{d}{dx}$ with expected absorption time  $1/r$, i.e. a solution to,
\begin{equation}
\label{eq.density.qsd}
\L^* v + c v'+ r v =0.
\end{equation}
\item $r\le  \Gamma(c)$,
where $\Gamma$ is the Legendre transform of the Laplace exponent of $\L$.
\item
A branching L\'evy process driven by $\L - c \frac{d}{dx}$, absorbed in $0$ gets almost surely
extinct.
\end{enumerate}

Moreover, $c$ is a minimal velocity for $(\L^*, r)$ if and only if $1/r$ is a minimal mean absorption time for $\L - c\frac{d}{dx}$.
\end{theorem}

\begin{remark}
In \eqref{eq.tw.levy} the domain is $\R$ and the boundary conditions are $w(+\infty)=1-w(-\infty)=1$, while in \eqref{eq.density.qsd} the domain is $(0,+\infty)$ and also $v\ge 0$, $v(0)=0$, $\int v =1$ is imposed.
\end{remark}

% 
% Finally, in Section \ref{sec:micro}, we introduce particle systems enlightening the selection principles
% observed for the macroscopic models. 

\section{Preliminaries} \label{sec:Levy}

%\subsection{L\'evy processes}

Let $X=(X_t)_{t\ge 0}$ be a L\'evy process with values in $\R$, defined on
a filtered space $(\Omega, \F, (\F_t),\P)$ and Laplace exponent
$\psi: \mathbb R \to  \mathbb R $ defined by
\[
\E(e^{\theta X_t})= e^{\psi(\theta ) t},
\]
such that
\[
\psi(\theta)= b\theta+ \sigma^2 \frac{\theta^2}{2} + g(\theta),
\]
where $b\in\R$, $\sigma>0$ (which ensures that $X$ is non-lattice) and $g$ is defined in terms of the jump measure $\Pi$ supported in $\R\setminus\{0\}$ by
\[
g(\theta)=\int_\R (e^{\theta x}- 1 -\theta x \one_{\{|x|<1\}}) \Pi(dx),  \qquad \int_\R (1 \wedge x^2) \Pi(dx) < \infty.
\]
Let $\theta_+^\star=\sup\{\theta\colon |\psi(\theta)|<\infty\}$, $\theta^\star_{-}=\inf\{\theta\colon |\psi(\theta)|<\infty\}$ and recall that $\psi$ is strictly convex in $(\theta^\star_-,\theta^\star_+)$ and by monotonicity $\psi(\theta^\star_\pm)=\psi(\theta^\star\mp)$ and $\psi'(\theta^\star_\pm)=\psi'(\theta^\star\mp)$ are well defined as well as the derivative at zero $\psi'(0)=\E(X_1)$, that we assume to be zero. We also assume that $\theta^\star_\pm >0$.  The generator of $X$ applied to a function $f\in C_0^2$, the class of compactly supported functions with continuous second derivatives, gives
\[
\L f(x)=\frac12 \sigma^2 f''(x) + bf'(x) + \int_\R (f(x+y) - f(x) - yf'(x)\one{\{|y|\le 1\}} ) \Pi(dy).
\]
The adjoint of $\L$ is also well defined in $C_0^2$ and has the form
\[
 \L^*f(x)=\frac12 \sigma^2 f''(x) - bf'(x) + \int_\R (f(x-y) - f(x) + yf'(x)\one{\{|y|\le 1\}} ) \Pi(dy).
\]
It is immediate to see that the Laplace exponent of $(X_t - ct)_{t \ge 0}$ is given by
$\psi_c(\theta)=\psi(\theta) - c \theta$ for $\theta \in [\theta^\star_-,\theta^\star_+	]$ and that
$C_0^2$ is contained in the domain of the generator $\L - c \frac{d}{dx}$. We denote by $\Gamma$ the Legendre transform of $\psi$, i.e., 
$$\Gamma(\alpha)= \sup_{\theta \in \mathbb R} \alpha \theta- \psi(\theta).$$
Similarly we will denote $\bar \Gamma$ the Legendre transform of the Laplace exponent of the dual process $(-X_t)_{t\ge 0}$,
$$\bar \Gamma(\alpha)= \sup_{\theta \in \mathbb R} \alpha \theta- \psi(-\theta).$$
Observe that since $\sigma>0$, $\Gamma$ as well as $\bar \Gamma$ are defined in $\R$. To summarize, hereafter we assume

\bigskip
\hspace{1cm}({\bf A}) $\sigma >0,$ $\theta^\star_\pm >0$ and $\E(X_1)=0$. 
\bigskip

Recall that the backward Kolmogorov equation for $X$ is given by
$$\frac{d}{dt} \E^x(f(X_t))= \L f(x),$$
while the forward Kolmogorov (or Fokker-Plank) equation for the density $u$ (which exists  since $\sigma>0$) is given by
$${d \over dt} u(t,x)= \L^*u (t,\cdot)(x).$$

We will consider on the one hand L\'evy processes with generator $\L$ (or $\L^*$) that evolve in $\R$ and on the other hand L\'evy processes with generator $\L -c \frac{d}{dx}$, killed at zero. A probability measure in $\R_+$ with density $v$ is a QSD for the process $(X_t -c t)_{t\ge 0}$ killed at $0$, if and only if, $v$ is a positive solution of \eqref{eq.density.qsd}. 

We will need the following.

\begin{lemma}[Girsanov theorem for L\'evy processes]\label{Girs.Levy}Let
$M^\theta_t:= \exp(\theta X_t - \psi_c(\theta)t)$ and the measure $\tilde\Q$ be defined by 
%\marginpar{\patu{Matt: me parece que la martingala deberia ser $M^\theta_t= \exp(\theta X_t - \psi_c(\theta)t) = \exp(\theta (X_t - ct) - \psi(\theta)t) $ pero no como esta ahora.}}
\begin{equation}
\label{change.of.measure}
 \frac{d\tilde\Q}{d\P}\Big|_{\F_t}= M^\theta_t, \qquad t\in [0,+\infty).
\end{equation}
Then $(M^\theta_t)_{t\ge 0}$ is a martingale and under $\tilde\Q$, $(X_t)_{t\ge 0}$ is a L\'evy process with drift $\E_{\tilde \Q}(X_1)=\psi'_c(\theta)= \psi'(\theta) -c$,
variance $\sigma^2$, and jump measure $ e^{\theta x} d\pi(x)$.
\end{lemma}

\subsection{Some useful results on branching L\'evy processes}

Consider a continuous time branching process with binary branching at rate $r>0$. Each individual performs 
independent L\'evy processes with generator $\L$ started at the position of his 
ancestor at her birth-time. Details on the construction of this process can be found in \cite{Ky}. 
Call $N_t$ the number of individuals in the process at time $t$ and $(\zeta_t^i, 
1\le i \le t)$ the positions of the individuals that are alive at time $t$. We 
call $Z_t=(\zeta^1_t,\dots, \zeta^{N_t}_t)$ and $Z=(Z_t)_{t\ge 0}$ a branching L\'evy process (BLP) driven by $\L$. 
%driven by $\L$. 
For some results, we need to consider BLP killed at some barrier $x\in\R$, the 
extension of the definition to this situation is straightforward.

The following proposition is proved in \cite{Biggins1,Biggins2}. See also \cite[Theorem 4.17]{BO} for an alternative proof with spines and a setting closer to ours.

\begin{prop}\label{prop:Rt}
Let $Z$ be a BLP driven by $\L$ and $R_t$ the position of the maximum of $Z_t$.
Then
$$ \lim_{t\to\infty}{R_t \over t} =\Gamma^{-1}(r).$$
\end{prop}

By means of this proposition we obtain the following partial extension of Theorem 1 in \cite{Bigginsetal91}. 

\begin{prop}\label{prop:branching}
Let $\mathring Z$ be a BLP driven by $\L - c {d \over dx}$ started at $x>0$ and killed at the origin. 
\begin{enumerate}
%\item For any interval $A\subset  \R$
%\[
% \lim_{t\to\infty} \frac{1}{t} \log \E\sum_{i=1}^{N_t}\one_{\{\zeta_t^{c,i} \in 
%A\}} = \log (r + \psi_c(\theta_c))
%\]
\item[(i)] If $r\le \Gamma(c)$, then $\mathring Z$ gets extinct with probability $1$. 
% Equivalently, the rightmost particle $R^c_t$ verifies $R^c_t  \to -\infty$ as 
% $t\to\infty$.

\item[(ii)] If $ r> \Gamma(c)$, then  for any interval $A\subset \R^+,$ $\P (\sum_{i=1}^{N_t}\one_{\{\mathring\zeta_t^i \in A\}} \to \infty) >0$.
\end{enumerate}
\end{prop}

\begin{proof}
Observe that $\mathring Z$ can be constructed straightforward with the trajectories of a non-absorbed process driven by the same generator. We just need to delete all the paths that touched the negative semi-axes at some time.
In the case $r < \Gamma(c)$, we can directly use the previous proposition to see that the maximum of the non-absorbed branching process satisfies ${R_t \over t} \to \Gamma^{-1}(r) - c <0$ which 
implies that $R_t$ is almost surely negative after some finite time. This in turn implies extinction of $\mathring Z$. For the critical case, we need to slightly refine the arguments given in \cite{BO}. 

Consider the branching L\'evy process $Z$ driven by $\L$ (without killing at $0$) defined in the same filtered space $(\Omega, \F, (\F_t),\P)$ and define the martingale
$$ Z^\theta_t= \sum_{i=1}^{N_t}\exp(\theta \zeta^i_t -(\psi_c(\theta) + r) t ),$$ 
%\marginpar{\patu{Mismo problema que antes con la martingala}}
as well as the change of measure,
$$
{d\Q \over d\P} \big|_{\F_t}= Z^{\theta}_t.
$$
On some suitably augmented filtration $\tilde \F_t \supset \F_t$, the new process can be seen as a branching process with a spine $(S_t)_{t\ge 0}$ which branches
at rate $2r$ and follows a motion given by the change of measure \eqref{change.of.measure},
i.e., a L\'evy process with drift $\psi'_c(\theta)= \psi'(\theta) - c$,
variance $\sigma^2$, and jump measure $ e^{\theta x} d\pi(x)$. The other particles follow the usual process $X$. See \cite{BO} for details on this construction.

Since we assumed $r=\Gamma(c)$, we can define $\theta_c$ such that $\psi_c(\theta_c)=-\Gamma(c)$ and so $\psi'_c(\theta_c)= 0$. From now on we choose $\theta=\theta_c$ in the change of measure and hence, the spine $(S_t)$ is centered.
As a consequence, it is recurrent (as a non trivial L\'evy process). It follows that  $\limsup_t S_t  =\infty $. Now bounding $Z^{\theta_c}_t$ by the contribution of the spine, we have 
$$ \limsup Z^{\theta_c}_t \ge \limsup \exp(\theta_c S_t -( \psi_c(\theta_c)+r )t )=  \exp(\theta_c S_t ).$$
Since $1/Z^{\theta_c}$ is a positive super-martingale (under $\Q$), it converges $\Q-$almost surely and so does $Z^{\theta_c}$.
Hence,
$$ \lim_{t\to\infty} Z^{\theta_c}_t= \infty, \quad \Q-\text{a.s.}$$
Observe that if $B\in \F_\infty$ we have
\[
 \Q(B) = \int_B \limsup_{t \to \infty} Z^\theta_t \, d\P + \Q(B\cap \{\limsup_{t \to \infty} Z^\theta_t = \infty \}).
\]
It then follows that if $ \lim Z^{\theta_c}_t= \infty$, under $\Q$,
then $ \lim Z^{\theta_c}_t= 0,$ under $\P$. Finally, let 
\[
R_t := \max_{1\le i\le N_t} \zeta^i_t -ct,
\]
and observe that $\exp(\theta_c R_t ) \le Z^{\theta_c}_t,$ which implies that $\exp(\theta_c R_t)$ tends to $0$ $\P-$a.s. and hence $R_t$ tends to $-\infty$.
As before, this implies extinction of $\mathring Z$.

\

To prove {\em (ii)}, denote $\mathring Z_t(A):= \sum_{i=1}^{N_t}\one_{\{\mathring \zeta_t^i \in A\}}$. We use the many-to-one lemma to get
\begin{equation}
\label{exp.zeta}
\E(\mathring Z_t(A)) = e^{rt}\P(X_t - ct \in A, \min_{0\le s \le t}X_s - cs \ge 0).
\end{equation}
To compute the last probability we can discretize the time variable and consider the random walk $S^\delta_n = X_{n\delta} - c \delta n$. Following \cite[Theorem 4]{TuominenTweedie} and \cite[Theorem 2.1]{Iglehart} we obtain that the decay parameter for the process $(X_t -ct)$ killed at zero is given by $\Gamma(c)$ and hence for every $r >\Gamma(c)$ the r.h.s of \eqref{exp.zeta} grows to infinity.
So, for any $x>0$ we can choose $t^*$ large enough to guarantee $\E^x(\mathring Z_{t^*}(A)) > 1$. Let $x=\inf A$. We can assume $x>0$ without loss of generality. Consider the (discrete time) Galton-Watson process with offspring distribution $\mathring Z_{t^*}(A)$, started with one individual at $x$. This process at time $n$ bounds from below $\mathring Z_{nt^*}(A)$ and since it is supercritical we have that $\mathring Z_{nt^*}(A)$ grows exponentially fast as $n\to\infty$ with positive probability.
Now,  
\begin{align*}
\P\Big(\mathring Z_s(A) \le \frac{\mathring Z_{nt^*}(A)}{2} \text{ for some } & nt^* \le s \le (n + 1)t^*\Big | \mathring Z_{nt^*}(A) \Big) \le \\
& \P^x(X_s - cs \le 0 \text{ for some } 0\le s \le t^*)^{\mathring Z_{nt^*}(A)/2}
\end{align*}
and the conditional Borel-Cantelli lemma \cite[p. 207]{Durrett1991} implies the result.

\end{proof}

\section{Quasi-stationary distributions and traveling waves}

In this section we prove the equivalence between existence of traveling waves and quasi-stationary distributions. The proof boils down to show that both are equivalent to the absorption of a BLP driven by $\L - c\frac{d}{dx}$ and killed at the origin.

\subsection{Existence of Quasi-stationary disributions}

We first deal with the quasi-stationary distributions.

\begin{prop}
\label{prop:QSD}
The following are equivalent
\begin{enumerate}
\item  There exists a QSD for $\L - c {d \over dx}$ killed at 0 with mean absorption time $1/r$.
\item $r\le \Gamma(c)$.
%\item $\lim\sup R_t \le 0$ where $R_t$ is the maximum of the L\'evy branching process with drift $c$ and branchinig rate $r$.
\end{enumerate}
\end{prop}

\begin{remark}
The existence of a QSD for $r=\Gamma(c)$ has been established in \cite{KP1} under stronger assumptions on the L\'evy process. 
%The existence of QSD's in the range $r< \Gamma(c)$ as well as the non-existence for $r>  \Gamma(c)$ is, up to our knowledge, new. 
\end{remark}

\begin{proof}

\vspace{.1cm}

{\em 1) $\Longrightarrow$ 2) (Non-existence).} Assume there exists a non-trivial QSD $\nu$ and suppose $r >  \Gamma(c)$. Since $\sigma>0$, 
there necessarily exists a density $v$ being the Radon-Nikodym derivative of $\nu$ with respect to the Lebesgue measure on $\mathbb R_+$.
Note that $v(0)=0$ and on $\mathbb R_+$ we have
$$ \L^* v + c v' +r v=0.$$
Let $\mathring{ \bar Z} =(\mathring{\bar \zeta}^1_t , \dots, \mathring{\bar \zeta}^{N_t}_t )$ be a branching L\'evy process driven by $\L^* + c\frac{d}{dx}$  killed at $0$ and started at $x>0$.
The process
$$M_t = \sum_{i=1}^{N_t} v(\mathring{\bar \zeta}_i(t) ),$$
is a martingale. 
% Since $M_t$ is positive it converges to an a.s finite limit $M_\infty$ sucht that $E^x(M_\infty)=v(x)$. 
On the other hand, for every $A \subset \mathbb R^+$,
\begin{equation}
\label{Mt}
\E^x(M_t) \ge ( \inf_A v ) \E^x \sum_{i=1}^{N_t}\one_{\{\mathring{\bar \zeta}_t^{i}\in A\}} = ( \inf_A v ) e^{rt} \P^x  ( -X_t + ct \in A, \min_{0\le s \le t} -X_s + cs \ge 0 ). 
\end{equation}

% \marginpar{\matt{Here taking expectations, and using the many-to-one with large deviations is sufficient to conclude...}}
% \marginpar{\patu{Matt, creo que la prueba de la prop hace exactamente eso, pero hay que tener cuidado porque no alcanza con probar que las esperanzas tienden a infinito. Lo estoy dejando porque me parece que hace falta}}

We want to show that the r.h.s in \eqref{Mt} goes to infinity. Observe that if we take $A=\R^+$ we know the asymptotic behavior of the probability on the r.h.s of \eqref{Mt}, but since $\inf_{R^+} v=0$ this is useless. So we need to choose a smaller $A$. Irreducibility implies that $\inf_A v >0$ for every $A \subset \R_+$ bounded and at a positive distance from the origin. We are going to choose $A=[\frac 1n,n]$ for an adequate $n>0$.
Consider the process $X^n=(X^n_t)_{t\ge 0}$ with generator $\L - c\frac{d}{dx}$ killed at $\frac 1 n$ and $n$ and call $p_n(x,t,B)=\P^x (X^n_t \in B)$ the transition semigroup and $\lambda_n$ its decay parameter (\cite[Theorem 6]{TuominenTweedie}) such that for every interval $B$
\[
- \lim_{t \to \infty} \frac 1 t  \log p_n(x,t,B)= \lambda_n.
\]
We use $p_\infty, \lambda_\infty$, etc. when we deal with the process in $\R^+$ killed at the origin. We will show that $\lambda_n \searrow \lambda_\infty=\Gamma(c)$ and hence, since $r>\Gamma(c)$ we can choose $n$ such that $r-\lambda_n>0$ and the r.h.s of \eqref{Mt} goes to infinity. A contradiction to the fact that $M_t$ is a martingale. Here we are using the fact that the exit problem from $[\frac 1 n, n]$ for a process with generator $\L^* + c\frac{d}{dx}$ started at $x$ is equivalent to the exit problem from the same interval for a process with generator $\L - c\frac{d}{dx}$ started at $y=n-x + \frac 1 n$.

Since $(\lambda_n)$ is decreasing in $n$, we only need to show $\lim \lambda_n \le \lambda_\infty$. By means of time-discretization, using the splitting technique (which allows us to assume that $X_t^n$ has an atom) 
and the subadditive ergodic theorem \cite[Section 4]{TuominenTweedie}, it can be shown that there exists a sequence of times $t_k \nearrow \infty$, $\ve >0$ and a constant $c>0$, both depending on $x$ and $\ve$ but not on $n$ such that
\[
-\frac{1}{t_k} \log p_n(y,t_k,(y-\ve,y+\ve)) + \frac{c}{t_k} \ge \lambda_n.
\]
For fixed $t_k<\infty$ we can take $n\to \infty$ to obtain
\[
-(1/t_k) \log p_\infty(y,t_k,B) + \frac{c}{t_k} \ge  \lim_{n\to\infty}\lambda_n.
\]
Now we let $k \to \infty$ to get $\lambda_\infty \ge \lim_n \lambda_n$. The fact that $\lambda_\infty=\Gamma(c)$ was already shown in the course of the proof of Proposition \ref{prop:branching}.

%\paragraph{Non-existence of a TW for $r > - \psi_c(\theta_c)$.}
%We follow the same argument as Harris \cite{H}.

%Define $F(c)=- \psi_c(\theta_c)$. Using the properties of $\Psi$, $F$ is  strictly decreasing  and defining $c*= F^{-1}(r)$,  $r > - \psi_c(\theta_c)$ implies that  $c < c ^*$

\medskip

{\em 2) $\Longrightarrow$ 1) (Existence)}. As before, note that $\nu$ is a QSD with density $v$ if and only if
\begin{equation}
\label{eq.int.density.qsd}
\int f( \L^* v + c v' + r v) =0,
\end{equation}
for all $f \in \mathcal D$ where $\mathcal D$ is a subset of the domain of the generator with killing, i.e. the original generator but with domain composed of functions vanishing at $0$, with the property that for every measurable set $A \subset \mathbb R_*^+$, there exists a sequence
$f_n$ in $\mathcal D$, uniformly bounded and converging pointwise to $1_A$.
Let $\theta>0$ and denote by $e_{-\theta}$ the function $x \mapsto e^{-\theta x}$ and $v(x)=e^{-\theta x}h(x)$. The function $h\colon \R_{\ge 0} \to \R$ will be determined later. Let $(X_t)_{t\ge 0}$ be a L\'evy process with generator $\L$. We compute
\begin{align*}
\left( \L^*+c\frac{d}{dx}\right) v(x) &= \frac{d}{dt}\E^0\left[e^{-\theta(x -X_t + ct)}h(x-X_t + ct)\right]_{t=0},\\
 &= e^{-\theta x}\frac{d}{dt}\left[ e^{\psi(\theta)t} \E^{0} \left(  e^{\theta X_t -(c\theta + \psi(\theta))t}h(x-X_t+ct)\right)\right]_{t=0}\\
& = e^{-\theta x}\frac{d}{dt}\left[ e^{(\psi(\theta) - c\theta)t} \tilde \E^{0} \left(h(x-X_t+ct)\right)\right]_{t=0}
\\
& = e^{-\theta x}\left( (\psi(\theta) - c\theta )h(x) + \tilde{\mathcal L} h(x) \right).
\end{align*}
here $\tilde \E$ denotes expectation under the measure $\tilde \Q$ defined by \eqref{change.of.measure} and $\tilde \L$ is the generator of a L\'evy process with drift $\tilde \E(\tilde X_1)= \tilde \E(- X_1) = c-\psi'(\theta)$, variance $\sigma^2$ and jump measure $ e^{-\theta x} d\pi(-x)$ as in Lemma \ref{Girs.Levy}. Hence
$$
\L^* v+ c v'+ rv = e^{-\theta x}\left( {(\psi(\theta) - c\theta + r )h(x)} + {\tilde \L} h(x) \right).
$$

We obtained that \eqref{eq.int.density.qsd} is equivalent to the following equation
$$ \int f e_{-\theta}( \tilde \L   h + (r+ \psi_c(\theta))h) =0,$$

Note that since $\psi$ is a convex function and $-\Gamma(c) \le -r$, it is possible to choose $\theta$ such that $\psi(\theta)- c \theta=-r$. Hence  \eqref{eq.int.density.qsd} is equivalent to 
$$ \int f  \tilde \L h  =0,$$
for all $f \in e_{-\theta} \mathcal D= \{ g = e_{-\theta} u, u \in \mathcal D \}$.
We then look for harmonic functions for the killed L\'evy process $\tilde X$ with generator
$\tilde \L$.

Define the renewal measure associated to $\tilde X$
$$ h(x) = \E \int_0^\infty \one_{\{\tilde H_t \ge x\}} dt,$$
where $\tilde H= (\tilde H_t)_{t\ge 0}$ is the ladder process associated to $-\tilde X$.

Let $\theta_c$ be defined by $\Gamma(c)=c\theta_c- \psi(\theta_c)$. For $\theta \le \theta_c$, the process $\tilde X$ does not drift to $-\infty$, since
$$\tilde \E^0(\tilde X_1)= c- \psi'(\theta) = -\psi_c'(\theta) \ge 0.$$
This implies that the function $h$ is harmonic (see Lemma 1 in \cite{CD}) and since moreover $\tilde X_1$ has a finite mean, the renewal theorem implies that  $h$ is asymptotically equivalent to the identity and so
$$ \int e_{-\theta} h < \infty.$$
Then $v$ is the density of a QSD with absorption rate $r$.  
\end{proof}

\subsection{Existence of Traveling waves}
We now present the corresponding equivalence for the case of traveling waves which was actually the inspiration for the
equivalence presented previously. Let us underline that these results are already known except in the critical case $r = \Gamma(c)$, \cite{Ky}. The proof is included for completeness but follows the proof of \cite{SH} who himself quote the results of Neveu \cite{neveu}.
%The main difficulty hence boiled down to prove the absorption property of the associated branching process, which was done in the previous section, in Proposition \ref{prop:branching}.

\begin{prop}[\cite{neveu,SH}]
\label{prop:KPP}
The following are equivalent
\begin{enumerate}
\item  There exists a solution to \eqref{eq.tw.levy}.
\item $r\le \Gamma(c)$.
%\item The branching L\'evy process driven by the process with generator $\L - c {d \over dx }$, branching rate $r$ and killed at zero gets almost surely extinct. 
\end{enumerate}
\end{prop}

\begin{proof}

{\em 1) $\Longrightarrow$ 2) (Non-existence).}
Assume the existence of a non-trivial traveling wave $w$. This allows us to define the multiplicative positive martingale
$$M_t= \prod_{i=1}^{N_t} w (\bar\zeta^{i}_t + ct).$$
Here $\bar Z_t=(\bar\zeta^1_t, \dots,\bar\zeta^{N_t}_t)$ is a BLP driven by $\L^*$ (with no killing).
This martingale being positive and bounded, it converges and its mean being $w(x)$, its limit is not $0$.
On the other hand, since $w \le 1$,
$$M_t \le w(\bar L_t + ct),$$
where $\bar L_t=\min_{1\le i\le N_t}\bar \zeta^{i}_t$.
Remark that the minimum of a BLP driven by $\L^*$ has the same law as $-\max_{1\le i\le \bar N_t}\zeta^{i}_t$,
where $Z= ((\zeta^i_t)_{1\le i\le N_t})_{t\ge 0}$ is a BLP driven by $\L$. Proposition \ref{prop:Rt} implies that if $r> \Gamma(c)$, $R_t -ct = \max_{1\le i\le \bar N_t}\zeta^{i}_t  - ct \to +\infty$. So
$\bar L_t   \to -\infty$ and hence $M_t$ should have a null limit. A contradiction to the the assumption.

{\em 2) $\Longrightarrow$ 1) (Existence).} Neveu's method for proving the existence of traveling waves consists in constructing a multiplicative martingale from a Galton-Watson process obtained as follows.
 
Consider $\mathring Z$ a BLP driven by $\L - c {d \over dx}$ with killing at the origin and started with one individual at $x>0$ as in Proposition \ref{prop:branching}. Since $r\le \Gamma(c)$ the process is absorbed and then the total population size is finite a.s. We can construct this random number for every $x>0$ using a unique BLP $\mathring{\bar Z}$ with generator $\L^* + c \frac{d}{dx}$ started with one individual at the origin and killing (freezing) at $x$. If we couple all the processes in this way and call $G_x <\infty$ the number of individuals of $\bar Z$ that have reached high $x$, we get that $(G_x)_{x\ge 0}$ is a continuous-time Galton-Watson process, \cite{neveu, SH}. Define
$$f_x(s)= \E(s^{G_x}),$$
and for some fixed $s \in (0,1)$
$$w(x)= f_x^{-1}(s).$$
Note that both quantities are strictly positive since $G_x < \infty$.  For $y\ge 0$ define
$$M^x_y:= w(x+y)^{G_y}.$$
It turns out that $(M^x_y)_{y\ge 0}$ is a convergent martingale. To see that, observe that the branching property gives us
\begin{align*}
\E[M^x_{y'} | \F_y]&=  \E[w(x+y')^{ G_{y'}} | \F_y],\\
&= (f_{y'-y}(w(x+y')))^{ G_{y}},\\
&= (f_{y'-y}(f_{y'-y}^{-1}(w(x+y))))^{G_y}= M^x_y.
\end{align*}
In addition, $(M^x_y)_{y\ge 0}$ is positive and bounded and hence, it does converge and is uniformly integrable. Following the arguments of \cite{SH}, for fixed $t$ and for all $y$ large enough
$$ G_y= \sum_{k=1}^{N_t} G^{i}_{y-\bar\zeta^i_t},$$
where the $(G^i)_{i\ge 1}$ are independent copies of $G=(G_x)_{x\ge 0}$. Hence
\[
M^x_y = \prod_{i=1}^{N_t} w(x+y)^ {G^{i}_{y-\bar\zeta^i_t}} = \prod_{i=1}^{N_t} M^{x+\bar\zeta^i_t,i}_{y-\bar\zeta^i_t}.
\]
and the limit of the martingale satisfies
$$
M^x = \prod_{i=1}^{N_t} M^{x+\bar\zeta^i_t,i}.
$$
taking expectations leads to
$$ w(x) = \E\prod_{i=1}^{N_t} w(x+\bar\zeta^i_t),$$
which in turn implies (see Theorem 8  in \cite{Ky}) that
$$
\L^* w + c w' + r(w^2-w)=0. 
$$

\end{proof}

% \begin{remark}
% This constructive proof of existence presented for F-KPP traveling waves could be also considered for the
% the existence of QSDs replacing multiplicative martingales by additive martingales.
% Unfortunately the convergence of the so-called Biggins (or derivative in the context of random walks)
% martingale, to a non-zero limit is much more difficult to prove than for the multiplicative one
% and is out of the scope of this paper.
% \end{remark}

\subsection{Proof of Theorem \ref{theo:main}}

Observe that Proposition \ref{prop:KPP} gives us $1) \iff 3)$ while Proposition \ref{prop:QSD} proves $2) \iff 3)$. The equivlence  $3) \iff 4)$ is the content of Proposition \ref{prop:branching}. Finally since $\Gamma$ is strictly increasing, minimality of $1/r$ (for a given $c$) as well as minimality of $c$, for a given $r$, reduces to 
\[
 r=\Gamma(c).
\]

\bibliographystyle{plain}
\bibliography{biblio2}

\end{document}